\documentclass [11pt]{article}
\usepackage[]{amsmath,amssymb,amsfonts,latexsym,amsthm,enumerate}
\usepackage{mathrsfs}
\usepackage[numeric,initials,nobysame]{amsrefs}

\long\def\symbolfootnote[#1]#2{\begingroup%
\def\thefootnote{\fnsymbol{footnote}}\footnote[#1]{#2}\endgroup}

\def\R{\mathbb{R}}

\newif\ifhide
\newif\ifexpose
\hidefalse
\exposetrue

\def\GG{{G}}

\def\expec#1{{\mathbb E[#1]}\,}

\def\prob#1{\textrm{Pr}[#1]}

\date{}
\title{Sub-Gaussian tails for the number of triangles in $\GG(n,p)$}
\author{{Guy Wolfovitz\thanks{Department of Computer Science, 
Haifa University, Haifa, Israel. Email address:
{\tt gwolfovi@cs.haifa.ac.il}.}}}
\newtheorem{theorem}{Theorem}[section]
\newtheorem{lemma}[theorem]{Lemma}

\newtheorem{definition}{Definition}

\newtheorem{question}[theorem]{Question}

\renewcommand{\epsilon}{\varepsilon}

\DeclareMathOperator{\var}{Var}

\newtheoremstyle{upright}%
        {8pt plus2pt minus4pt}%
        {8pt plus2pt minus4pt}%
        {\upshape}%
        {}%
        {\bfseries}%
        {:}%
        {1em}%
        {}%
\theoremstyle{upright}

\newcommand{\ignore}[1]{}
\begin{document}

\maketitle

\begin{abstract}
Let $X$ be the random variable that counts the number of triangles in the
random graph $\GG(n,p)$.  We show that for some absolute constant $c$, the
probability that $X$ deviates from its expectation by at least $\lambda
\var(X)^{1/2}$ is at most $e^{-c\lambda^2}$, provided that $n^{-1}(\ln n)^{10}
\le p \le n^{-1/2}(\ln n)^{-10}$, $\lambda = \omega(\ln n)$ and $\lambda \le
\min\{(np)^{1/2}, n^{-3/4}p^{-3/2}, n^{1/6}\}$.
\end{abstract}

\section{Introduction} \label{sec:1}
%
In this paper we consider the standard Erd\H{o}s-R\'enyi random graph
$\GG(n,p)$, in which every edge of $K_n$ appears independently with probability
$p$.  We study the number of triangles, denoted by $X$, in $\GG(n,p)$.  This is
a classical topic in the theory of random graphs~\cites{ER60, Sc79, B81, KR83,
JLR87, Ru88, F89, Vu01, Vu02, JR02, KV04, JOR04, K90}. Our starting point is
the following question regarding the distribution of $X$.  This question has
been explicitly raised and studied by Vu~\cites{Vu01,Vu02} and more recently by
Kannan~\cite{K90}.
\begin{question} For which $p$ and $\lambda$ does $X$ have the sub-Gaussian
tails
\begin{eqnarray} \label{eq:q} \prob{|X - \expec{X}| \ge \lambda \var(X)^{1/2}}
\le e^{-c\lambda^2}, \end{eqnarray}
where $c$ is an absolute positive constant?  \end{question}

Ruci\'nski~\cite{Ru88} showed that if $1/2 \ge p = \omega(n^{-1})$ then
$\frac{X - \expec{X}} {\var(X)^{1/2}}$ tends in distribution to the normal
distribution $N(0,1)$.  This implies that~(\ref{eq:q}) holds for the same range
of $p$ for every constant~$\lambda \ge 0$. Vu~\cites{Vu01} showed that there is
a constant $c_1 > 0$ such that~(\ref{eq:q}) holds for $p = \omega(n^{-1/2}\ln
n)$ and $c_1np^2 \ge \lambda = \omega(\ln n)$.  Vu~\cite{Vu02} also showed for
every constant $c_2 > 0$ there is a constant $c_3 > 0$ such that~(\ref{eq:q})
holds if $p \ge n^{-1/2 + c_2}$ and $0 \le \lambda \le n^{c_3}$.  Recently,
Kannan~\cite{K90} showed that there are constants $c_4, c_5 > 0$ such that if
$c_4n^{-1}\ln n \le p \le c_4n^{-3/5}$, then for $\lambda = O(np)$, $\prob{|X -
\expec{X}| \ge \lambda \var(X)^{1/2}} \le c_5 e^{-c\lambda}$ for some absolute
constant $c$.  
In this paper we improve upon Kannan's result both by expanding the range of
$p$ and by giving a better upper bound on the tail.  More importantly, our
result complements in a way Vu's results, in that it addresses the question
above with regard to the case where $n^{-1}(\ln n)^{10} \le p \le n^{-1/2}(\ln
n)^{-10}$.  Formally, we prove the following.
\begin{theorem}
\label{thm:main}
Inequality~(\ref{eq:q}) is valid if $n^{-1}(\ln n)^{10} \le p \le n^{-1/2}(\ln
n)^{-10}$, $\lambda = \omega(\ln n)$ and $\lambda \le \min\{(np)^{1/2},
n^{-3/4}p^{-3/2}, n^{1/6}\}$.
\end{theorem}
The proof of Theorem~\ref{thm:main} employs an iterative invocation of
McDiarmid's inequality (which is stated at the next subsection), and a certain
iterative view of the random graph $\GG(n,p)$. The proof is given in the next
section.

\subsection{McDiarmid's inequality}
%
Let $\alpha_1, \alpha_2, \ldots, \alpha_m$ be independent random variables with
$\alpha_i$ taking values in a set $A_i$. Let $f : \prod_{i=1}^m A_i \to \R$
satisfy the following Lipschitz condition: if two vectors $\alpha, \alpha' \in
\prod_{i=1}^m A_i$ differ only in the $i$th coordinate, then $|f(\alpha) -
f(\alpha')| \le a_i$.  McDiarmid's inequality~\cite{McD} states that the random
variable $W = f(\alpha_1, \alpha_2,\ldots, \alpha_m)$ satisfies for any $t \ge
0$,
\begin{eqnarray*} \prob{|W - \expec{W}| \ge t} \le
2\exp\bigg(-\frac{2t^2}{\sum_{i=1}^m a_i^2}\bigg).  \end{eqnarray*}
%

\section{Proof of Theorem~\ref{thm:main}} \label{sec:2}
%
Fix $p$ and $\lambda$ within the ranges asserted by the theorem. It is safe to
assume, and we will use this implicitly in the proof, that $n \ge n_0$ for some
sufficiently large constant $n_0$ which we do not explicitly state (otherwise
the theorem is trivial).  The proof of the theorem relies on the analysis of
the following iterative process, which gives an alternative definition of
$\GG(n,p)$:
\begin{definition} \label{def} Let $\epsilon \le 1/1000$ be a constant such that
$\epsilon^I = p$ for some integer $I \le \ln n$. Let $G_0 := K_n$.  Given
$G_i$, $i \ge 0$, construct $G_{i+1}$ by taking every edge in $G_i$
independently with probability~$\epsilon$. End upon obtaining $G_I$.
\end{definition}
It is clear that $G_i$ has the same distribution as $\GG(n,\epsilon^i)$.  In
particular, by the definition of~$I$, $G_I$ has the same distribution as
$\GG(n,p)$.  Let $X_i$ be the random variable that counts the number of
triangles in $G_i$ and note that $X = X_I$.  Let $Y_{i,e}$ be the number of
sets $\{e', e''\} \subseteq G_i$ such that $\{e, e', e''\}$ is a triangle.  Let
$Z_{i,e} := Y_{i,e} \cdot {\bf 1}[e \in G_i]$, where ${\bf 1}[e \in G_i]$ is
the indicator function for the event that $e \in G_i$. (In words, $Z_{i,e}$ is
equal to $Y_{i,e}$ if $e \in G_i$ and is equal to $0$ otherwise.)  We will use
$r \pm s$ below to denote the interval $[r-s, r+s]$. The following lemma, as we
soon show, can be easily used to prove the theorem. 
\begin{lemma} \label{lemma}
There is a constant $c_0 > 0$ s.t. for all $0 \le i < I$ the following holds.
Assume that
\begin{itemize} 
\item $X_i \in \binom{n}{3} \epsilon^{3i} \pm \big(\binom{n}{3}\epsilon^{3i} i
(np)^{-3/2} + 0.1 \lambda \sqrt{n^3\epsilon^{3i}}\big)$.  
\item $\forall e \in K_n. \,\,\,\, Y_{i,e} \le \max\{4 n \epsilon^{2i} +
\lambda\sqrt{4n\epsilon^{2i}}, \lambda^2\}$.  
\item $\sum_{e \in K_n} Z_{i,e}^2 \le n^4 \epsilon^{5i} (1 + in^{-1/4}) + 10
\binom{n}{3}\epsilon^{3i}$.  
\end{itemize}
Then each of the following items occurs with probability at least
$1-2e^{-c_0\lambda^2}$.
\begin{itemize} 
\item $X_{i+1} \in \binom{n}{3} \epsilon^{3(i+1)} \pm \big(\binom{n}{3}
\epsilon^{3(i+1)} (i+1) (np)^{-3/2} + 0.1 \lambda
\sqrt{n^3\epsilon^{3(i+1)}}\big)$.  
\item $\forall e \in K_n. \,\,\,\, Y_{i+1,e} \le \max\{4 n \epsilon^{2(i+1)} +
\lambda \sqrt{4n\epsilon^{2(i+1)}}, \lambda^2\}$. 
\item $\sum_{e \in K_n} Z_{i+1,e}^2 \le n^4 \epsilon^{5(i+1)} (1 +
(i+1)n^{-1/4}) + 10 \binom{n}{3}\epsilon^{3(i+1)}$.  
\end{itemize}
\end{lemma}

\begin{proof}[Proof of Theorem~\ref{thm:main}]
The preconditions in Lemma~\ref{lemma} hold trivially for $i=0$.  Since $I \le
\ln n$ and $\lambda = \omega(\ln n)$, we thus get from Lemma~\ref{lemma} that
with probability at least $(1 - 6e^{-c_0\lambda^2})^{I} \ge 1 - e^{-0.5c_0
\lambda^2}$, 
\begin{eqnarray*}
X = X_I \in \binom{n}{3}{\epsilon^{3I}} \pm \bigg(\binom{n}{3}\epsilon^{3I}
I(np)^{-3/2} + 0.1 \lambda\sqrt{n^3\epsilon^{3I}} \bigg) \subseteq \expec{X}
\pm 0.2 \lambda (np)^{3/2},
\end{eqnarray*}
where the last containment follows since $\expec{X} = \binom{n}{3}p^3$ and $p =
\epsilon^{I}$, and from the upper and lower bounds on $I$ and~$\lambda$
respectively.  This implies the validity of Theorem~\ref{thm:main}, as one can
easily verify that $\var(X) \ge 0.1(np)^3$ for our choice of $p$.
\end{proof}

It remains to prove Lemma~\ref{lemma}. Fix a constant $c_0 > 0$, sufficiently
small so that it satisfies our claims below.  Fix $0 \le i < I$ and assume that
we are given $G_i$ and that the preconditions in the lemma hold for $i$. We
show that each of the three consequences in the lemma holds with probability at
least $1 - 2e^{-c_0 \lambda^2}$. 

{\bf First consequence.} Clearly, $$\expec{X_{i+1}} = \epsilon^3 X_i \in
\binom{n}{3}\epsilon^{3(i+1)} \pm \bigg( \binom{n}{3} \epsilon^{3(i+1)} i
(np)^{-3/2} + 0.1 \epsilon^{1.5}\lambda\sqrt{n^3\epsilon^{3(i+1)}}\bigg).$$ We
need to bound from above the probability that $X_{i+1}$ deviates from its
expectation by more than $t_1 := \binom{n}{3} \epsilon^{3(i+1)} (np)^{-3/2} +
0.1 (1- \epsilon^{1.5}) \lambda \sqrt{n^3\epsilon^{3(i+1)}}$.  Every edge $e
\in G_i$ has an outcome which is either the event that $e \in G_{i+1}$ or not.
Note that $X_{i+1}$ depends on the outcomes of the edges in $G_i$. Also note
that changing the outcome of a single edge $e \in G_i$ can change $X_{i+1}$ by
at most $Z_{i,e}$ and that $\sum_{e \in G_i} Z_{i,e}^2 = \sum_{e \in K_n}
Z_{i,e}^2$.  Therefore, by McDiarmid's inequality and by the assumed upper
bound on $\sum_{e \in K_n} Z_{i,e}^2$,
\begin{eqnarray} \label{eq:q1}
\prob{|X_{i+1} - \expec{X_{i+1}}| > t_1} \le 2 \exp\bigg(-\frac{t_1^2}{\sum_{e
\in K_n} Z_{i,e}^2}\bigg) \le 2 \exp\bigg(-\frac{t_1^2}{2n^4\epsilon^{5i} +
2n^3\epsilon^{3i}}\bigg).
\end{eqnarray}
We have two cases. The first case is that $\epsilon^{i} \ge n^{-1/2}$.  In that
case $2n^4\epsilon^{5i} + 2n^3\epsilon^{3i} \le 4 n^4\epsilon^{5i}$.  Using
this and the fact that $t_1 \ge 0.1n^3 \epsilon^{3(i+1)} (np)^{-3/2}$, we get
from~(\ref{eq:q1}) that $$\prob{|X_{i+1} - \expec{X_{i+1}}| > t_1} \le 2
\exp\bigg(-\frac{ 0.01 n^6 \epsilon^{6(i+1)} (np)^{-3}}{4
n^4\epsilon^{5i}}\bigg) \le 2 \exp\bigg(-\frac{c_0 \epsilon^i}{np^3}\bigg) \le
2 e^{-c_0 \lambda^2}, $$ where the last inequality follows from the fact that
$\epsilon^i \ge n^{-1/2}$ and $\lambda \le n^{-3/4} p^{-3/2}$.

Next consider the case that $\epsilon^i < n^{-1/2}$.  In that case,
$2n^4\epsilon^{5i} + 2n^3\epsilon^{3i} \le 4 n^3\epsilon^{3i}$.  Since $t_1 \ge
0.05\lambda \sqrt{n^3\epsilon^{3(i+1)}}$, we get from~(\ref{eq:q1}) that
$$\prob{|X_{i+1} - \expec{X_{i+1}}| > t_1} \le 2 \exp\bigg(-\frac{ 0.0025
\lambda^2 n^3 \epsilon^{3(i+1)}}{4 n^3 \epsilon^{3i}}\bigg) \le 2 e^{-c_0
\lambda^2}.$$

{\bf Second consequence.} Fix $e \in K_n$. If $Y_{i,e} \le \lambda^2$ then
clearly $Y_{i+1,e} \le Y_{i,e} \le \lambda^2$ with probability~$1$. Otherwise
we have $Y_{i,e} > \lambda^2$.  Hence by assumption, $Y_{i,e} \le
4n\epsilon^{2i} + \lambda \sqrt{4n\epsilon^{2i}}$.  This implies that
$\expec{Y_{i+1,e}} \le 4n\epsilon^{2(i+1)} + \epsilon \lambda
\sqrt{4n\epsilon^{2(i+1)}}$. Note that $4n\epsilon^{2i} \ge 0.25 \lambda^2$,
since otherwise $Y_{i,e} \le \lambda^2$. Hence
$\expec{Y_{i+1,e}} \le 4n\epsilon^{2(i+1)} + \epsilon
\lambda\sqrt{4n\epsilon^{2(i+1)}} \le 12 n \epsilon^{2(i+1)}$. This in turn
implies, again using $4n\epsilon^{2i} \ge 0.25 \lambda^2$, that
$\expec{Y_{i+1},e} + O(\lambda \sqrt{4n\epsilon^{2(i+1)}}) =
O(n\epsilon^{2(i+1)})$. Therefore, by Chernoff's bound, the probability that
$Y_{i+1,e}$ deviates from its expectation by more than $(1-\epsilon) \lambda
\sqrt{4n\epsilon^{2(i+1)}}$ is at most $e^{-2c_0 \lambda^2}$. Thus, for a fixed
$e \in K_n$, we have that $Y_{i+1,e} \le \max\{4n \epsilon^{2(i+1)} +
\lambda\sqrt{4 n \epsilon^{2(i+1)}}, \lambda^2\}$ with probability at least $1
- e^{-2c_0 \lambda^2}$.  It now follows from the union bound and the fact that
$\lambda = \omega(\ln n)$ that with probability at least $1 - n^2
e^{-2c_0\lambda^2} \ge 1 - 2 e^{-c_0 \lambda^2}$, $Y_{i+1,e} \le
\max\{4n\epsilon^{2(i+1)} + \lambda \sqrt{4n\epsilon^{2(i+1)}}, \lambda^2\}$
holds for all $e \in K_n$ simultaneously, as needed.

{\bf Third consequence.} We start by estimating $\expec{Z_{i+1,e}^2}$ from
above for a fixed $e \in K_n$.  Clearly, if $e \notin G_i$ then
$\expec{Z_{i+1,e}^2} = Z_{i,e} = 0$. So assume $e \in G_i$.  If $e \notin
G_{i+1}$ then trivially $\expec{Z_{i+1,e}^2} = 0$. Conditioning on the event
that $e \in G_{i+1}$, $Z_{i+1,e}$ is a bonimial random variable with mean
$\epsilon^2 Z_{i,e}$ and variance $\epsilon^2 (1-\epsilon^2) Z_{i,e}$.
Therefore, conditioning on $e \in G_{i+1}$, we have that $\expec{Z_{i+1,e}^2} =
\expec{Z_{i+1,e}}^2 + \var(Z_{i+1,e}) \le \epsilon^4 Z_{i,e}^2 + \epsilon^2
Z_{i,e}$.  Adding the fact that the event $e \in G_{i+1}$ occurs with
probability $\epsilon$ we can conclude that without the conditioning on $e \in
G_{i+1}$, $\expec{Z_{i+1,e}^2} \le \epsilon^5 Z_{i,e}^2 + \epsilon^3 Z_{i,e}$.

Let $Z := \sum_{e \in K_n}  Z_{i+1,e}^2$. By linearity of expectation and the
previous paragraph, $$\expec{Z} = \sum_{e \in K_n} \expec{Z_{i+1,e}^2} \le
\sum_{e \in K_n} \epsilon^5 Z_{i,e}^2 + \epsilon^3 Z_{i,e}.$$ Every triangle in
$G_i$ is counted exactly $3$ times in the sum $\sum_{e \in K_n} Z_{i,e}$ and so
$\sum_{e \in K_n} Z_{i,e} = 3 X_i$. Also, $X_i \le 2\binom{n}{3}
\epsilon^{3i}$, and this follows from the assumed estimate on $X_i$ and the
bounds on $p,~\lambda$~and~$I$.  This implies that $\sum_{e \in K_n} \epsilon^3
Z_{i,e} = 3 \epsilon^3 X_i \le 6 \binom{n}{3} \epsilon^{3(i+1)}$. Using this,
the above upper bound on $\expec{Z}$ and the assumed upper bound on $\sum_{e
\in K_n} Z_{i,e}^2$ we get that \begin{eqnarray*} \expec{Z} &\le& n^4
\epsilon^{5(i+1)} ( 1 + in^{-1/4} ) + 10 \binom{n}{3} \epsilon^{3i + 5} + 6
\binom{n}{3} \epsilon^{3(i+1)} \\ &\le& n^4 \epsilon^{5(i+1)} ( 1 + in^{-1/4} )
+ 9 \binom{n}{3} \epsilon^{3(i+1)}.  \end{eqnarray*}

It remains to estimate from above the probability that $Z$ deviates from its
expectation by more than $t_2 := n^4\epsilon^{5(i+1)}n^{-1/4} + \binom{n}{3}
\epsilon^{3(i+1)}$.  Clearly $Z$ depends on the outcome of the edges in $G_i$.
Fix $e \in G_i$ and let $\sum_{\{e', e''\}}$ be the sum over all $Z_{i,e}$ sets
$\{e', e''\}$ such that $\{e, e', e''\}$ is a triangle in $G_i$.  We claim that
changing the outcome of~$e$ can change $Z$ by at most \begin{eqnarray*}
Z_{i,e}^2 + \sum_{\{e',e''\}} (2Z_{i,e'} + 2Z_{i,e''} + 2) & = & Z_{i,e}
Y_{i,e} + \sum_{\{e', e''\}} (2Y_{i,e'} + 2Y_{i,e''} + 2) \\ &\le&
Z_{i,e}Y_{i,e} + 4 Z_{i,e} \cdot \max\{ 4n\epsilon^{2i} + \lambda
\sqrt{4n\epsilon^{2i}}, \lambda^2\} + 2Z_{i,e} \\ &\le& Z_{i,e}Y_{i,e} + 6
Z_{i,e} \cdot \max\{ 4n\epsilon^{2i} + \lambda \sqrt{4n\epsilon^{2i}},
\lambda^2\} \\ &\le& 7 Z_{i,e} \cdot \max\{4n\epsilon^{2i} + \lambda
\sqrt{4n\epsilon^{2i}}, \lambda^2\}.  \end{eqnarray*} Indeed, if $e \notin
G_{i+1}$ then $Z_{i+1,e} = 0$ and otherwise $Z_{i+1,e} \le Z_{i,e}$. Hence,
changing the outcome of~$e$ can change $Z_{i+1,e}^2$ by at most $Z_{i,e}^2$. In
addition, for every triangle $\{e, e', e''\}$ in $G_i$, changing the outcome
of~$e$ can change $Z_{i+1,e'}$ and $Z_{i+1,e''}$ each by at most~$1$. Since
$Z_{i+1,e'} \le Z_{i,e'}$, this implies that changing the outcome of~$e$ can
change $Z_{i+1,e'}^2$ by at most $(Z_{i,e'}+1)^2 - Z_{i,e'}^2 \le 2 Z_{i,e'} +
1$. The same argument also shows that changing the outcome of $e$ can change
$Z_{i+1,e''}^2$ by at most $2 Z_{i,e''} + 1$. Lastly note that changing the
outcome of $e$ can affect only the sum $Z_{i+1,e}^2 + \sum_{\{e',e''\}}
Z_{i+1,e'}^2 + Z_{i+1,e''}^2$.

If $4n\epsilon^{2i} \le \lambda^2$ then $\max\{4n\epsilon^{2i} +
\lambda\sqrt{4n\epsilon^{2i}}, \lambda^2\} \le 2\lambda^2$. If on the other
hand $4n\epsilon^{2i} > \lambda^2$ then $\max\{4n\epsilon^{2i} +
\lambda\sqrt{4n\epsilon^{2i}}, \lambda^2\} \le 8n\epsilon^{2i}$.  Hence
$\max\{4n\epsilon^{2i} + \lambda \sqrt{4n\epsilon^{2i}}, \lambda^2\}$ is at
most $\max\{8n\epsilon^{2i}, 2\lambda^2\}$. Also note that $\sum_{e \in G_i}
Z_{i,e}^2 = \sum_{e \in K_n} Z_{i,e}^2$. Therefore, given the discussion above,
it follows from McDiarmid's inequality and the assumed upper bound on $\sum_{e
\in K_n} Z_{i,e}^2$ that \begin{eqnarray} \label{eq:q2} \nonumber \prob{|Z -
\expec{Z}| > t_2} &\le& 2 \exp\bigg(-\frac{t_2^2}{\sum_{e \in G_i} (7 Z_{i,e}
\cdot \max\{ 8n\epsilon^{2i}, 2\lambda^2 \})^2}\bigg) \\ &\le& 2
\exp\bigg(-\frac{t_2^2}{100 \max\{64n^2\epsilon^{4i}, 4\lambda^4\} \cdot
(n^4\epsilon^{5i} + n^3\epsilon^{3i})} \bigg). \end{eqnarray} 
Assume that $\epsilon^i \ge n^{-1/2}$.  In that case, $n^4\epsilon^{5i} +
n^3\epsilon^{3i} \le 2n^4\epsilon^{5i}$. In addition, trivially $t_2 \ge
n^4\epsilon^{5(i+1)}n^{-1/4}$.  Therefore, from~(\ref{eq:q2}) it follows that
\begin{eqnarray*} \prob{|Z - \expec{Z}| > t_2} &\le& 2
\exp\bigg(-\frac{n^8\epsilon^{10(i+1)}n^{-1/2}} {100 \max\{64n^2\epsilon^{4i},
4\lambda^4\} \cdot 2n^4\epsilon^{5i}}\bigg) \\ &\le& 2
\exp\bigg(-\frac{c_0n^{3.5}\epsilon^{5i}}{\max\{n^2\epsilon^{4i},
\lambda^4\}}\bigg) \\ &\le& 2 e^{-c_0\lambda^2},\end{eqnarray*} where the last
inequality follows since $\epsilon^i \ge n^{-1/2}$ and $\lambda \le n^{1/6}$.

Next assume that $\epsilon^i < n^{-1/2}$.  In that case, $\max\{64 n^2
\epsilon^{4i}, 4 \lambda^4\} = 4 \lambda^4$ and $n^4\epsilon^{5i} +
n^3\epsilon^{3i} \le 2n^3\epsilon^{3i}$. In addition, trivially, $t_2 \ge 0.1
n^3\epsilon^{3(i+1)}$. Therefore, from~(\ref{eq:q2}),
\begin{eqnarray*} \prob{|Z - \expec{Z}| > t_2} &\le& 2 \exp\bigg(-\frac{0.01
n^6 \epsilon^{6(i+1)}} {400 \lambda^4 \cdot 2n^3 \epsilon^{3i}}\bigg) \\ &\le&
2 \exp\bigg(- \frac{c_0 n^3 \epsilon^{3i}}{\lambda^4}\bigg) \\ &\le& 2
e^{-c_0\lambda^2}, \end{eqnarray*} where the last inequality follows since
$\lambda \le (np)^{1/2} \le (n\epsilon^i)^{1/2}$.

\section{Concluding remarks}
Combining Theorem~\ref{thm:main} and a result of Vu~\cite{Vu01}, we have that
for every $p \ge n^{-1}(\ln n)^{10}$, if $p$ does not satisfy $n^{-1/2}(\ln
n)^{-10} < p \le O(n^{-1/2}\ln n)$, then one can take some~$\lambda =
\lambda(n)$ that goes to~$\infty$ with~$n$ so that the probability that $X$
deviates from its expected value by at least $\lambda \var(X)^{1/2}$ is at most
$e^{-c\lambda^2}$, for some absolute constant $c$.  One question that remains
open is what happens when $p = n^{-1/2}$.  That is, can one show
that~(\ref{eq:q}) is valid for $p = n^{-1/2}$ and for some~$\lambda =
\lambda(n)$ that goes to $\infty$ with $n$?

In the proof of Theorem~\ref{thm:main} we had to assume that $\lambda =
\omega(\ln n)$. In fact we probably could have proved Theorem~\ref{thm:main}
had we assumed that $\lambda \ge C\ln n$ for some sufficiently large
constant~$C$.  However, our argument would have failed if we took $\lambda =
o(\ln n)$. This is because if $\lambda = o(\ln n)$, then Lemma~\ref{lemma} only
implies that with probability at least $1 - e^{-0.5c_0\lambda^2}$, $X \in
\expec{X} \pm \omega(\lambda) (np)^{3/2}$, and this does not imply the theorem.
This naturally raises the following question: is it true that~(\ref{eq:q})
holds for $n^{-1}(\ln n)^{10} \le p \le n^{-1/2}(\ln n)^{-10}$ and, say,
$\lambda = \sqrt{\ln n}$? 

Finally, we note that our argument for the proof of Theorem~\ref{thm:main} can
be generalized so as to prove a rather general concentration result for
functions with large Lipschitz coefficients.  This is the subject of a
forthcoming paper. This concentration result can provide some new sub-Gaussian
tail bounds for the number of copies of $H$ in $\GG(n,p)$, for a large family
of graphs $H$.


\end{document}